\documentclass[12pt]{amsart}
\usepackage[margin=1in]{geometry}
\usepackage{amsmath,amssymb,setspace,mathtools,hyperref}
\usepackage{indentfirst}
\usepackage{comment}
\usepackage{todonotes}
\DeclareFontFamily{OT1}{rsfs}{}
\DeclareFontShape{OT1}{rsfs}{n}{it}{<-> rsfs10}{}
\DeclareMathAlphabet{\mathscr}{OT1}{rsfs}{n}{it}

\newcommand{\triv}{\text{triv}}

\newtheorem{theorem}{Theorem}
\newtheorem{corollary}[theorem]{Corollary}
\newtheorem{lemma}[theorem]{Lemma}
\newtheorem{proposition}[theorem]{Proposition}
\newtheorem{definition}[theorem]{Definition}
\newtheorem{remark}[theorem]{Remark}
\newtheorem{example}[theorem]{Example}
\numberwithin{theorem}{section}

\DeclareMathOperator{\sgn}{sgn}
\DeclareMathOperator{\Succ}{Succ}
\DeclareMathOperator{\Pred}{Pred}
\DeclareMathOperator{\Sep}{Sep}
\begin{document}

\title [Parameters for Generalized Hecke Algebras in Type $B$]{Parameters for Generalized Hecke Algebras in Type $B$} \author{Max Murin and Seth Shelley-Abrahamson}\date{\today}\maketitle

\begin{abstract}  The irreducible representations of full support in the rational Cherednik category $\mathcal{O}_c(W)$ attached to a Coxeter group $W$ are in bijection with the irreducible representations of an associated Iwahori-Hecke algebra. Recent work has shown that the irreducible representations in $\mathcal{O}_c(W)$ of arbitrary given support are similarly governed by certain generalized Hecke algebras.  In this paper we compute the parameters for these generalized Hecke algebras in the remaining previously unknown cases, corresponding to the parabolic subgroup $B_n \times S_k$ in $B_{n+k}$ for $k \geq 2$ and $n \geq 0$.\end{abstract}

\tableofcontents

\section{Introduction}

Let $W$ be a finite Coxeter group with complex reflection representation $\mathfrak{h}$.  Etingof and Ginzburg \cite{EG} defined for each such pair $(W, \mathfrak{h})$ a family of infinite-dimensional associative algebras $H_c(W, \mathfrak{h})$, the \emph{rational Cherednik algebras}, depending on a parameter $c$ and deforming the algebra $\mathbb{C}W \ltimes D(\mathfrak{h})$, where $D(\mathfrak{h})$ is the algebra of polynomial differential operators on $\mathfrak{h}$.  Since their introduction, rational Cherednik algebras and their representation theory have received much attention due to their connections with many other topics in mathematics, including, for example, Hilbert schemes \cite{GS1, GS2}, torus knots \cite{EGL}, quantum integrable systems \cite{F}, categorification \cite{Shan, SV}, and the representation theory of more classical algebras such as Iwahori-Hecke algebras.

Of particular interest to representation theorists is the category $\mathcal{O}_c = \mathcal{O}_c(W, \mathfrak{h})$ \cite{GGOR} of certain admissible representations of $H_c(W, \mathfrak{h})$.  The category $\mathcal{O}_c$ may be viewed as an analogue of the classical Bernstein-Gelfand-Gelfand category $\mathcal{O}$ of representations of a semisimple complex Lie algebra $\mathfrak{g}$.  In parallel with the Verma modules appearing in the classical categories $\mathcal{O}$, to each irreducible representation $\lambda$ of $W$ there is a \emph{standard module} $\Delta_c(\lambda)$ in $\mathcal{O}_c$, and each standard module $\Delta_c(\lambda)$ has a unique irreducible quotient, denoted $L_c(\lambda)$.  This association $\lambda \mapsto L_c(\lambda)$ determines a bijection $\text{Irr}(W) \rightarrow \text{Irr}(\mathcal{O}_c)$ between the irreducible representations of the Coxeter group $W$ and the irreducible representations in the category $\mathcal{O}_c$.  In fact, for almost all values of the parameter $c$ this association defines an equivalence of categories, while at certain special values of $c$ the category $\mathcal{O}_c$ becomes non-semisimple.  In this sense, the family of categories $\mathcal{O}_c$ depending on $c$ deforms the category of finite-dimensional complex representations of $W$ in an interesting way.

Each representation $M$ of $H_c(W, \mathfrak{h})$ in the category $\mathcal{O}_c$ has an associated \emph{support} $\text{Supp(M)}$, a closed subvariety of the vector space $\mathfrak{h}$.  In a sense, $\text{Supp}(M)$ measures the size of $M$.  For example, at one extreme the condition $\text{Supp}(M) = \{0\}$ is equivalent to the finite-dimensionality of $M$.  At the other extreme, the condition $\text{Supp}(M) = \mathfrak{h}$, i.e. that $M$ has \emph{full support}, means that $M$ is as large as possible, i.e. that the Hilbert polynomial of $M$ has degree equal to the dimension of $\mathfrak{h}$.  It is a fundamental question about the category $\mathcal{O}_c$ to ask how many irreducible representations have a given support $X \subset \mathfrak{h}$; in the case $X = \{0\}$ this is precisely asking for the number of finite-dimensional irreducible representations of $H_c(W, \mathfrak{h})$.  The case $X = \mathfrak{h}$ was treated in \cite{GGOR}, in which it was shown that the set of irreducible representations in $\mathcal{O}_c$ of full support are naturally in bijection with the set of irreducible representations of the (better understood) Iwahori-Hecke algebra $H_q(W)$ of $W$ at a parameter $q$ depending on $c$ in an exponential manner.

More recently, Losev and the second author \cite{LS} extended this result from \cite{GGOR} to allow for the counting of irreducible representations with any given support.  In particular, it is shown in \cite[Theorem 4.2.14]{LS} that the representations in $\mathcal{O}_c$ of any given support are governed, in an appropriate sense, by generalized Hecke algebras of the form $G_1 \ltimes H_q(G_2)$.  Here, $G_2$ is a finite Coxeter group, $H_q(G_2)$ is the Hecke algebra of $G_2$ at a particular parameter $q$, and $G_1$ is a finite group acting on $H_q(G_2)$ by diagram automorphisms.  In particular, the counting problem is reduced to the determination of these groups $G_1, G_2$ and the parameter $q$ in all cases.  This situation is in complete analogy with the Harish-Chandra series and generalized Hecke algebras appearing in the representation theory of finite groups of Lie type.

The groups $G_1$ and $G_2$ are easily determined by the support $X \subset \mathfrak{h}$ in question, but the parameter $q$ is much subtler.  In \cite[Theorem 4.2.11, Remark 4.2.12]{LS} a procedure is given for computing the parameter $q$, although computationally this procedure is quite complicated.  In particular, there is a parameter $q(W_1, W_2, c, \lambda) \in \mathbb{C}^\times$ associated to each tuple $(W_1, W_2, c, \lambda)$ such that
\begin{enumerate}
\item $W_2$ is an irreducible finite Coxeter group \item $W_1 \subset W_2$ is a standard parabolic subgroup of rank one less than the rank of $W_2$ \item $W_1 \subsetneq N_{W_2}(W_1)$, where $N_{W_2}(W_1)$ denotes the normalizer of $W_1$ in $W_2$ \item $c$ is a complex-valued class function on the set of reflections in $W_2$ \item $\lambda$ is an irreducible representation of $W_1$ such that $\dim_\mathbb{C} L_c(\lambda) < \infty.$
\end{enumerate}
Given a tuple $(W_1, W_2, c, \lambda)$ satisfying (1)-(5) as above, computing $q(W_1, W_2, c, \lambda)$ by the procedure in \cite[Theorem 4.2.11]{LS} amounts to computing a complicated element of the Hecke algebra $H_{q_{W_1}}(W_1)$, where the parameter $q_{W_1}$ is given by $q_{W_1}(s) = e^{-2 \pi i c(s)}$, and taking its character value in an appropriate representation depending on $\lambda$.  The parameters of the Hecke algebra $H_q(G_2)$ are given by the numbers $q(W_1, W_2, c, \lambda)$ for appropriate values of $W_1, W_2, c, \lambda$.  These parameters $q(W_1, W_2, c, \lambda)$ were computed in \cite{LS} in all cases except when $(W_1, W_2) = (B_n \times S_k, B_{n + k})$, where $n, k$ are integers satisfying $n \geq 0$ and $k \geq 2$.

In this paper, we compute the parameters $q(W_1, W_2, c, \lambda)$ for the generalized Hecke algebras in these remaining previously unknown cases appearing in type $B$.  In particular, we show that $q(B_n \times S_k, B_{n + k}, c, \lambda) = -(-e^{-2\pi i c(s)})^k$, where $s \in B_{n}$ is a reflection through any of the coordinate axes in $\mathbb{C}^{n}$.  Remarkably, we find that this parameter does not depend on $n$ or on the representation $\lambda \in \text{Irr}(B_n \times S_k)$, extending a pattern observed in \cite[Remark 4.2.15]{LS} in other cases.  We prove this result by induction on $n$, in which the base case $n = 0$ is established by producing a combinatorial formula for the square of a certain element of the Hecke algebra $H_{p, q}(B_n)$ in the standard $T_w$-basis.

\subsection{Acknowledgments}

This work was conducted as part of the 2017 MIT Summer Program in Undergraduate Research.  We would like to thank David Jerison and Ankur Moitra for numerous useful discussions.  We also thank the creators of the CHEVIE package for GAP 3 (\cite{Mi}, \cite{MHLMP}), which was extremely helpful in discovering several of our results and verifying the correctness of this paper.

\section{Background and Definitions}\label{sec:background}

\subsection{Finite Coxeter Groups}

\begin{definition}
Let $S$ be a finite set, and, for each pair of distinct elements $s, t \in S$, let $m_{s,t} = m_{t,s} \geq 2$ be a positive integer or $\infty$.  Let $W$ be the group generated by $S$ with the following relations:

\begin{enumerate}
	\item $s^2 = 1$ for each $s \in S$
	\item $\underbrace{sts\cdots }_{m_{s,t}} = \underbrace{tst\cdots }_{m_{s,t}}$ for each pair of distinct elements $s, t \in S$ such that $m_{st} < \infty$.
\end{enumerate}

A group $W$ with such a presentation is called a \emph{Coxeter group}, and the tuple $(W, S)$ is called a \emph{Coxeter system}. Given a Coxeter system $(W, S)$, the set $S$ is called the set of \emph{simple reflections}.
\end{definition}

Throughout this paper, we will only consider finite Coxeter groups, i.e. the case when $|W| < \infty$.

Attached to any Coxeter system is the corresponding length function.

\begin{definition}
Let $(W, S)$ be a Coxeter system. Then, for any $w \in W$, the \emph{length} of $w$, denoted $\ell(w)$, is the length of the shortest sequence $(s_{i_1}, \ldots, s_{i_r})$ of simple reflections such that $s_{i_1}\cdots s_{i_r} = w$. A product of simple reflections equaling $w$ of length $\ell(w)$ is called a \emph{reduced expression}.
\end{definition}

Note that if $s_{i_1}\cdots s_{i_k}$ is a reduced expression, then any contiguous subword $s_{i_l}s_{i_{l+1}}\cdots s_{i_m}$ is also a reduced expression. Note also that if $w = s_{i_1}\cdots s_{i_k}$, then $w^{-1} = s_{i_k}\cdots s_{i_1}$, and thus $\ell(w^{-1}) = \ell(w)$ for any $w$.

A finite Coxeter group $W$ has a unique element of maximal length, denoted $w_0$ \cite[1.5.1]{GP}. For any element $w \in W$, $\ell(ww_0) = \ell(w_0w) = \ell(w_0) - \ell(w)$. This implies that $\ell(w_0^2) = \ell(w_0) - \ell(w_0) = 0$, and thus $w_0^2 = 1$.

Given a Coxeter group $W$ with simple reflections $S$ and a subset $S' \subset S$, then the (standard) \emph{parabolic subgroup} $W'$ of $W$ generated by $S'$ is the group generated by the elements of $S'$. The pair $(W', S')$ is a Coxeter system with $m'_{s,t} = m_{s, t}$ for $s, t \in S'$.  For any element $w \in W'$, any reduced expression for $w$ contains only elements of $W'$. Thus, the length function on $W'$ is the restriction of the length function on $W$ to $W'$ \cite[1.2.10]{GP}.

Given a Coxeter system $(W, S)$ and a parabolic subgroup $W'$, there are distinguished representatives of the $W'$-cosets in $W$:

\begin{definition}
Let $W'$ be a parabolic subgroup of $W$. Then, for each $w \in W$ there exists a unique element $x$ of minimal length in the coset $W'w$ \cite[2.1.1]{GP}. This is the \emph{distinguished right coset representative} of $W'w$. Let $X_{W'} \subset W$ be the set of such representatives.
\end{definition}

For any $w' \in W'$ and $x \in X_{W'}$, we have that $\ell(w'x) = \ell(w') + \ell(x)$.  The set of distinguished left coset representatives is defined similarly, and $w \in W$ is a distinguished right coset representative for $W'$ if and only if $w^{-1}$ is a distinguished left coset representative for $W'$.

\begin{example}
Let $S_k$, for any $k \geq 2$, be the $k^{th}$ symmetric group, i.e., the group of permutations of the set $\{1, \ldots, k\}$. This group is realized as a Coxeter group with simple reflections $s_i = (i, i+1)$ for $1 \leq i \leq k-1$ and with $m_{s_i, s_j} = 2$ if $|i - j| > 1$, and $m_{s_i, s_{i+1}} = 3$. The length of an element $w \in S_k$ is equal to the number of inversions, $|\{1 \leq i < j \leq k \mid w(i) > w(j)\}|$. The longest element in this group is the permutation $i \mapsto k+1-i$.
\end{example}

\begin{example}
Let $B_n$, for any $n \geq 1$, be the group of permutations $w$ of $\{\pm 1, \ldots, \pm n\}$ such that $w(-i) = -w(i)$ for each $i \in \{\pm 1, \ldots, \pm n\}$. Then, $B_n$ is a Coxeter group, with simple reflections $t = (-1, 1)$ and $s_i = (-i, -(i+1))(i, i+1)$ for each $1 \leq i \leq k-1$. We have $m_{t, s_1} = 4$, $m_{t, s_i} = 2$ for each $2 \leq i \leq k-1$, $m_{s_i, s_{i+1}} = 3$ for each $1 \leq i \leq n-2$, and $m_{s_i, s_j} = 2$ if $1 \leq i, j \leq n-1$ and $|i-j| > 1$.
\end{example}

For $x$ a nonzero integer, let $\sgn(x)$ be $1$ if $x > 0$ and $-1$ if $x < 0$.  The following lemma, giving an interpretation of the length of an element $w \in B_n$ in terms of the associated signed permutation, is standard:

\begin{lemma}\label{lem:b-length}
With respect to the choice of simple reflections above, the length of an element $w \in B_n$ is given by
\[
\ell(w) = |\{(i, j) \mid 1 \leq i < j \leq n, w(i) > w(j)\}| + \sum_{\substack{1 \leq i \leq n \\ w(i) < 0}} |w(i)|.
\] 
\end{lemma}

\begin{proof}
Let $e_1, \ldots, e_n$ be the standard basis of $\mathbb{R}^n$.  The group $B_n$ acts faithfully on $\mathbb{R}^n$ by $we_i = \sgn(w(i)) e_{|w(i)|}$ for each $w \in B_n$ and $1 \leq i \leq n$, and this representation of $B_n$ is its standard reflection representation as a Coxeter group.  The associated root system $\Phi$ is given by $$\Phi =  \{\pm e_i : 1 \leq i \leq n\} \cup \{\pm e_i \pm e_j : 1 \leq i < j \leq n\}.$$  The system of positive roots $\Phi^+ \subset \Phi$ given by $$\Phi^+ = \{e_i : 1 \leq i \leq n\} \cup \{\pm e_i + e_j : 1 \leq i < j \leq n\}$$ is compatible with the choice of simple reflections $t, s_1, \ldots, s_{n - 1}$.  In particular, for any $w \in B_n$, $\ell(w) = |w\Phi^+ \cap -\Phi^+|$ \cite[1.3.4]{GP}.

Now, note that for $1 \leq i \leq n$ we have $w(e_i) = \sgn(w(i))e_{|w(i)|} \in -\Phi^+$ if and only if $w(i) < 0$.  Also, for any $1 \leq i < j \leq n$, $w(-e_i+e_j) = -\sgn(w(i))e_{|w(i)|} + \sgn(w(j))e_{|w(j)|} \in -\Phi^+$ if and only if  $w(i) > w(j)$.  Finally, consider $w(e_i + e_j)$. If $|w(i)| < |w(j)|$, this is positive if and only if $w(j)$ is positive. Thus, 
\[
\ell(w) = |\{i \mid 1 \leq i \leq n, w(i) < 0\}| + |\{(i, j) \mid 1 \leq i < j \leq n, w(i) > w(j)\}| + \sum_{\substack{1 \leq i \leq n \\ w(i) < 0}} \left(|w(i)|-1\right).
\]
Combining the first and third terms gives the desired formula.\end{proof}

For $w \in B_n$, we will denote by $-w$ the element $w_0w = ww_0$ given by $(-w)(i) = -w(i)$.

\subsection{Hecke Algebras}

\begin{definition}
Let $(W, S)$ be a Coxeter system, and let $q : S \to \mathbb{C}$ be a function satisfying $q(wsw^{-1}) = q(s)$ for all $w \in W$ and $s \in S$. Define the \emph{Hecke algebra} $H_q(W)$ to be the $\mathbb{C}$-algebra with generators $T_s$, $s \in S$, and relations
\begin{enumerate}
	\item $\underbrace{T_sT_tT_s\cdots }_{m_{s,t}} = \underbrace{T_tT_sT_t\cdots }_{m_{s,t}}$ for each pair of distinct elements $s, t \in S$
	\item $(T_s - 1)(T_s + q(s)) = 0$ for each $s \in S$.
\end{enumerate}
\end{definition}

\begin{remark}
We choose a different convention than \cite[8.1.4]{GP}, which uses the relation $(T_s' + 1)(T_s' - q(s)) = 0$ instead of relation (2). These definitions give isomorphic algebras (for example, under the mapping $T_s \mapsto -T_s'$) after appropriately modifying the parameter $q$.
\end{remark}

\begin{proposition}{\cite[4.4.3]{GP}}\label{prop:t-basis}
Let $(W, S)$ be a Coxeter system. For each $w$, there is a well-defined element $T_w \in H_q(W)$ such that
\[
T_w = T_{s_1}\ldots T_{s_l}
\]
whenever $w = s_1\ldots s_l$, $s_i \in S$, is a reduced expression. Additionally, for any $w \in W$, $s \in S$, 
\[
T_wT_s = \left\{\begin{array}{ll} T_{ws} & \text{ if $\ell(ws) > \ell(w)$} \\ q(s)T_{ws} + (1-q(s)) T_{w} & \text{ otherwise.}\end{array}\right.
\]
An analogous statement holds for $T_sT_w$. The set $\{T_w \mid w \in W\}$ forms a $\mathbb{C}$-basis for $H_q(W)$.
\end{proposition}

The $T$-basis here also gives an alternate presentation for the Hecke algebra as the algebra generated by generators $T_w$, $w \in W$, with relations given by the formula for $T_wT_s$ appearing in Proposition \ref{prop:t-basis} for any $w \in W$ and $s \in S$.

\begin{proposition}{\cite[4.4.7]{GP}}\label{prop:free-module}  Let $W'$ is a parabolic subgroup of Coxeter group $W$ and let $X_{W'}$ be the set of distinguished right coset representatives of $W'$ in $W$. Then, as a left-$H_q(W')$ module,
\[
H_q(W) \cong \bigoplus_{x \in X_{W'}} H_q(W')T_x,
\]
and this decomposition is compatible with the $T$-bases of $H_q(W)$ and $H_q(W')$.
\end{proposition}

For the remainder of this paper, we will consider only Hecke algebras attached to Coxeter groups $B_n$ and $S_k$. In $S_k$, all simple reflections are conjugate, so the Hecke parameter $q$ must take on the same value for all $s \in S$. Thus, we will write $H_q(S_k)$, $q \in \mathbb{C}$, to mean the Hecke algebra with parameter $s \mapsto q$. In $B_n$, all $s_i$ are conjugate, but $t$ falls in a separate conjugacy class. Thus, we will write $H_{p, q}(B_n)$, $p, q \in \mathbb{C}$, to indicate the Hecke algebra with parameter $t \mapsto p, s_i \mapsto q$.

\subsection{Computing Parameters in Type \texorpdfstring{$B$}{B}}

\begin{definition}
For integers $n \geq 0$ and $k \geq 2$, let $X_{n, k}$ be the set of distinguished coset representatives of $B_n \times S_k$ in $B_{n+k}$. Let $w_{n, k}$ denote the element of maximal length in $X_{n, k}$.
\end{definition}

We have that $w_{n, k} = w_0w_0' = w_0'w_0$, where $w_0$ is the longest element of $B_{n+k}$, and $w_0'$ is the longest element of $B_n \times S_k$. Note that $w_{n, k}$ is an involution.

Let $\mathbb{C}_\triv$ be the trivial representation of $H_q(S_k)$ on $\mathbb{C}$ in which each basis element $T_w$ acts by multiplication by 1, and let $1_\triv := 1 \in \mathbb{C}_\triv$.  We regard $\mathbb{C}_\triv$ as a right module over $H_q(S_k)$.  The main result of this paper is the following:

\begin{theorem}\label{thm:main}
Let $n \geq 0$ and $k \geq 2$, and let $q$ be a primitive $k^{th}$ root of unity. Let $z_x \in H_{p, q}(B_n \times S_k)$, $x \in X_{n, k}$, be the unique coefficients so that
\[
T_{w_{n, k}}^2 = \sum_{x \in X_{n, k}} z_xT_x.
\]
Then, $1_\triv \otimes z_{w_{n, k}} = 1_\triv \otimes (1 + (-p)^k)T_1$ as elements of $\mathbb{C}_\triv \otimes_{H_q(S_k)} H_{p, q}(B_n \times S_k)$.
\end{theorem}

Theorem \ref{thm:main} completes the determination of the parameters for generalized Hecke algebras appearing in the representation theory of rational Cherednik algebras attached to Coxeter groups in the following way.  Consider any tuple $(B_n \times S_k, B_{n+k}, c, L)$ meeting the conditions in the introduction.  These are the tuples for which the associated generalized Hecke algebra parameter $Q := q(B_n \times S_k, B_{n + k}, c, L)$ is unknown.  Let $c_1 = c(t)$ and $c_2 = c(s_1)$. Without loss of generality, to compute $Q$, we may assume that $c_2 = \frac{r}{k}$ with $r$ a positive integer relatively prime to $k$ and that $L = L_c(\lambda \otimes \triv)$, where $\lambda$ is an irreducible representation of $B_n$ and $\triv$ is the trivial representation of $S_k$; see \cite[Section 4.4]{LS} for details. Then, by Theorem \ref{thm:main} and \cite[Theorem 4.2.11]{LS}, the matrix
\[
T = \left(\begin{array}{cc} 0 & -(-p^k) \\ 1 & 1 + (-p)^k\end{array}\right),
\]
where $p = e^{-2i\pi c_1}$, satisfies the quadratic relation $(T-1)(T+Q) = 0$. Thus, $Q = -(-p)^k$, with no dependence on $n$ or $L$.

The remainder of this paper is dedicated to studying the element $T_{w_{n, k}}^2$ and proving Theorem \ref{thm:main}.

\section{Case of \texorpdfstring{$S_k \subset B_k$}{S\_k in B\_k}}\label{sec:base-case}

\subsection{Good Involutions and the \texorpdfstring{$T$}{T}-basis Decomposition of \texorpdfstring{$T_{w_{0, k}}^2$}{T\_w\_0,k squared}}
We will first prove a theorem giving a full description of $T_{w_{0, k}}^2$ in the $T$-basis in terms of combinatorial properties of a certain set of elements of $B_k$.

\begin{definition}\label{def:good}
Call an element $w \in B_k$ \emph{good} if $w^2 = 1$, and, for each $1 \leq i \leq k$, either $w(i) = i$ or $w(i) < 0$. Let $G_k$ denote the set of good involutions in $B_k$.
\end{definition}

We view the groups $B_k$ in a chain $B_1 \subset B_2 \subset B_3 \ldots$, where $B_i$ is identified as the subgroup of $B_{i+1}$ that fixes $i+1$. In this manner, $B_i$ is the parabolic subgroup of $B_{i+1}$ generated by $\{t, s_1, \ldots, s_{i-1}\}$. If $w \in G_k$, then $w$, viewed as an element of $B_{k+r}$ for any $r \geq 0$, is in $G_{k+r}$. Thus, the sets $G_k$ form a chain $G_1 \subset G_2 \ldots$ compatible with the chain $B_1 \subset B_2 \ldots$.

\begin{definition}\label{def:ad}
For any $w \in B_k$, let $a(w) = |\{j \mid 1 \leq j \leq k \mid w(j) = j\}|$, and, for $0 \leq i \leq k$, let $d(i, w) = |\{j \mid i < j \leq k, w(j) = j\}|$.
\end{definition}

\begin{definition}\label{def:tidy}
For any $w \in B_k$, call an unordered pair $\{i, j\}$, where $1 \leq i, j \leq k$ and $i \neq j$, \emph{tidy} in $w$, or $w$-\emph{tidy}, if $-w(i) < j$ and $-w(j) < i$. Denote by $c(w)$ the number of $w$-tidy pairs.
\end{definition}

\begin{remark}  With respect to the chain of inclusions of $B_k$, an element $w \in B_k$ determines an element $w_r \in B_{k+r}$, for any $r \geq 0$. However, $a(w_r)$ depends on $r$; we have $a(w_r) = a(w) + r$.  Similarly, $d(i, w_r)$ and $c(w_r)$ depend on $r$.\end{remark}

We now present a formula for the decomposition of $T_{w_{0, k}}^2$ in the $T$-basis:
\begin{theorem}\label{thm:w0k-decomposition}
For each $k \geq 1$, we have:
\[
T_{w_{0, k}}^2 = \sum_{w \in G_k} p^{\frac{k+a(w)-a(-w)}{2}}(1-p)^{a(-w)}q^{c(w)}(1-q)^{\frac{k-a(w)-a(-w)}{2}}T_w.
\]
\end{theorem}

Theorem \ref{thm:w0k-decomposition} will be proved after establishing some lemmas.

Let $x \in B_{k+1}$ be the element $ts_1s_2\cdots s_k$, and let $x_i$, for $1 \leq i \leq k$, be the element $ts_1s_2\cdots \hat s_i \cdots s_k$, where $\hat s_i$ indicates that it is left out of the product.

\begin{lemma}\label{prop:succ}
\ 
\begin{enumerate}
	\item If $w \in G_k$, then:
	\begin{itemize}
		\item $xwx^{-1} \in G_{k+1}$ and $xwx^{-1}(1) = 1$
		\item $xwx^{-1}t \in G_{k+1}$ and $xwx^{-1}t(1) = -1$
		\item For $1 \leq i \leq k$, $xwx_i^{-1} \in G_{k+1}$ if and only if $w(i) = i$, and in this case $xwx_i^{-1}(1) = -(i+1)$.
	\end{itemize}
	\item If $w \in G_{k + 1}$, then:
	\begin{itemize}
		\item $x^{-1}wx \in B_k$ if and only if $w(1) = 1$,
		\item $x^{-1}wtx \in B_k$ if and only if $w(1) = -1$,
		\item $x^{-1}wx_i \in B_k$ if and only if $w(1) = -(i+1)$ for $1 \leq i \leq k$.
	\end{itemize}
	In each such case, the resulting element of $B_k$ is good.
\end{enumerate}
\end{lemma}

\begin{proof}
First, consider $w \in G_k$. Then, $xwx^{-1}$ maps $1$ to $1$ and $i+1$ to $w(i) + \sgn(w(i))$ for $1 \leq i \leq k$, so this element is a good involution in $B_{k+1}$. Similarly, $xwx^{-1}t$ maps $1$ to $-1$ and $i+1$ to $w(i) + \sgn(w(i))$ for any $1 \leq i \leq k$, so this is also in $G_{k+1}$. Finally, consider $xwx_i^{-1}$ for some $1 \leq i \leq k$. For any $j \neq i$, $1 \leq j \leq k$, we have that $xwx_i^{-1}(j+1)= w(j) + \sgn(w(j))$, so for each such $j$, $xwx_i^{-1}(j+1) = j+1$ or $xwx_i^{-1}(j+1) < 0$. Also, $xwx_i^{-1}(i+1) = -1$, so $xwx_i^{-1}$ is good if and only if it sends $1$ to $-(i+1)$. But $xwx_i^{-1}(1) = -xw(i)$, and $x^{-1}(i+1) = i$, so $w(i) = i$ if and only if $xwx_i^{-1}$ is good, giving (1).

Now, consider $w \in G_{k+1}$. An element of $B_{k+1}$ is also an element of $B_k$ exactly if it fixes $k+1$. An element $x^{-1}wy$, $y \in B_{k+1}$, fixes $k+1$ if and only if $y(k+1) = wx(k+1) = -w(1)$. But $x(k+1) = -1$, $tx(k+1) = 1$, and $x_i(k+1) = i+1$. Thus, for each possible value of $w(1)$, there exists a unique element $y \in \{x, tx\} \cup \{x_i \mid 1 \leq i \leq k\}$ such that $x^{-1}wy$ fixes $k+1$. If $y$ is $x$ or $tx$, then $x^{-1}wy$ maps $j$ to $w(j+1) - \sgn(w(j+1))$ for each $1 \leq j \leq k$, and so $x^{-1}wy$ is good. If $y$ is $x_i$, then $x^{-1}wy$ maps $i$ to $i$, and for each $1 \leq j \leq k$ such that $j \neq i$, $x^{-1}wy(j) = w(j+1) - \sgn(w(j+1))$. Therefore, $x^{-1}wx_i \in B_k$ is good if $w(1) = -(i+1)$, giving (2).\end{proof}

\begin{definition}\label{def:succ}
For $w \in G_k$, define $\Succ(w)$ to be the set $\{xwx^{-1}, xwx^{-1}t\} \cup \{xwx_i^{-1} \mid 1 \leq i \leq k, w(i) = i\}$. For $w \in G_{k+1}$, define $\Pred(w)$ to be $x^{-1}wx$ if $w(1) = 1$, $x^{-1}wtx$ if $w(1) = -1$, and $x^{-1}wx_i$ if $w(1) = -(i+1)$ for $1 \leq i \leq k$.
\end{definition}

Note that for each $w \in G_k$, each $w' \in \Succ(w)$ has that $\Pred(w') = w$. Also, for each $w' \in G_{k+1}$, $w' \in \Succ(\Pred(w'))$. This shows that $G_{k+1} = \coprod_{w \in G_k} \Succ(w)$.

\begin{lemma}\label{lem:abc-succ}
For any $w \in G_k$, we have:
\begin{enumerate}
	\item $a(xwx^{-1}) = a(w) + 1$ and $a(-xwx^{-1}) = a(-w)$,
	\item $a(xwx^{-1}t) = a(w)$ and $a(-xwx^{-1}t) = a(-w) + 1$,
	\item $a(xwx_i^{-1}) = a(w) - 1$ and $a(-xwx_i^{-1}) = a(-w)$ for any $1 \leq i \leq k$ such that $w(i) = i$.
	\item $c(xwx^{-1}) = c(xwx^{-1}t) = c(w) + a(w)$.
	\item $c(xwx_i^{-1}) = c(w) + d(i, w)$ for any $1 \leq i \leq k$ such that $w(i) = i$.
\end{enumerate}
\end{lemma}

\begin{proof}
To show (1), recall that $xwx^{-1}$ sends $1$ to $1$ and sends $j+1$ to $w(j) + \sgn(w(j))$ for each $1 \leq j \leq k$. Thus, $xwx^{-1}(j+1) = j+1$ if and only if $w(j) = j$. Since $1$ is also fixed, $a(xwx^{-1}) = a(w) + 1$. The element $-xwx^{-1}$ sends $j+1$ to $-w(j) - \sgn(w(j))$ for each $1 \leq j \leq k$, and so $j+1$ is fixed in $-xwx^{-1}$ if and only if $j$ is fixed in $-w$. Additionally, $-xwx^{-1}$ does not fix $1$, so $a(-xwx^{-1}) = a(w)$.

For (2), note that $xwx^{-1}t = txwx^{-1}$.  In particular, $xwx^{-1}t$ again sends $j+1$ to $w(j) + \sgn(w(j))$ for each $1 \leq j \leq k$. However, this element sends $1$ to $-1$ instead. This gives us that $a(xwx^{-1}t) = a(w)$ and $a(-xwx^{-1}t) = a(-w) + 1$, as desired.

Next, consider $xwx_i^{-1}$ where $w(i) = i$. For each $1 \leq j \leq k$ other than $i$, this element sends $j+1$ to $w(j) + \sgn(w(j))$. Thus, for each such $j$, $xwx_i^{-1}$ fixes $j$ if and only if $w$ does. However, $w$ fixes $i$, and $xwx_i^{-1}$ fixes neither $1$ nor $i+1$, so $a(xwx_i^{-1}) = a(w) - 1$. Similarly, $a(-xwx_i^{-1}) = a(-w)$, proving (3).

As previously noted, both $xwx^{-1}$ and $xwx^{-1}t$ send $j+1$ to $w(j) + \sgn(w(j))$ for each $1 \leq j \leq k$. It follows that any pair $\{j, j'\}$, $1 \leq j, j' \leq k$, is tidy in $w$ if and only if $\{j+1, j'+1\}$ is tidy in $xwx^{-1}$, which is in turn equivalent to $\{j+1, j'+1\}$ being tidy in $xwx^{-1}t$. Now, consider a pair of the form $\{1, j+1\}$ in $xwx^{-1}$ or $xwx^{-1}t$. This pair is tidy if and only if $\pm 1 > -(j+1)$ and $w(j) + \sgn(w(j)) > -1$.  This holds if and only if $w(j) > 0$, which holds if and only if $w(j) = j$, because $w$ is good.  It follows that $c(xwx^{-1}) = c(xwx^{-1}t) = c(w) + a(w)$, proving (4).

Now, let $g = xwx_i^{-1}$ for some fixed $1 \leq i \leq k$ such that $w(i) = i$. For any $j \neq i$ with $1 \leq j \leq k$, $g(j + 1) = w(j) + \sgn(w(j))$. As before, for any pair $\{j, j'\}$ such that $1 \leq j, j' \leq k$, $j \neq j'$, and $j, j' \neq i$, we have that $\{j, j'\}$ is tidy in $w$ if and only if $\{j+1, j'+1\}$ is tidy in $xwx_i^{-1}$. It remains to check pairs of the form $\{1, j+1\}$, $\{i+1, j+1\}$, and $\{1, i+1\}$, for $j \neq i$ such that $1 \leq j \leq k$. Since $w(i) = i$, $w(i) > -j$ for each $j$ such that $1 \leq j \leq k$, and $g(i+1) = -1 > -(j+1)$ for each such $j$, and thus $\{i+1, j+1\}$ is tidy in $g$ if and only if $\{i, j\}$ is tidy in $w$. Additionally, $\{1, i+1\}$ is not tidy in $g$, as $g(i+1)$ is not greater than $-1$. Finally, consider pairs of the form $\{1, j+1\}$ for $1 \leq j \leq k$. That such a pair is tidy requires that $g(j+1) > -1$, which is equivalent to $w(j) = j$ because $w$ is good, and also requires that $g(1) = -(i+1) > -(j+1)$. Thus, this pair is tidy in $g$ if and only if $w(j) = j$ and $j > i$.  There are $d(i, w)$ such pairs by definition, proving (5).\end{proof}

\begin{lemma}\label{lem:conj-by-tx}
Let $w \in B_k$ be a good involution, and let $x = ts_1\cdots s_k \in B_{k + 1}$. Then,
\[
T_{x} T_w T_{x^{-1}} = pq^{a(w)}T_{xwx^{-1}} + (1-p)q^{a(w)}T_{xwx^{-1}t} + (1-q)\sum_{\substack{1 \leq j \leq k \\ w(j) = j}} q^{d(j, w)}T_{xwx_j^{-1}}.
\]
\end{lemma}

\begin{proof}
To prove this statement, we will prove, for $1 \leq i \leq k + 1$, that:
\begin{equation}\label{eqn:claim}
T_x T_w T_{s_ks_{k-1}\cdots s_i} = q^{d(i-1, w)}T_{xws_k\cdots s_i} + (1-q)\sum_{\substack{i \leq j \leq k \\ w(j) = j}} q^{d(j, w)} T_{xws_k\cdots \hat s_j \cdots s_i}.
\end{equation}

Let us prove Equation \ref{eqn:claim} by (downward) induction on $i$. For the base case, we take $i = k+1$, which requires that $T_xT_w = T_{xw}$. This follows from the statement that $\ell(xw) = \ell(x) + \ell(w)$ for any $w \in B_k$.

Now, assume inductively that Equation \ref{eqn:claim} holds for some value of $i$ with $2 \leq i \leq k+1$, and consider $T_xT_wT_{s_k\cdots s_i} T_{s_{i-1}}$. By the inductive hypothesis, this is equal to
\begin{equation}\label{eq:ind_step_partial_product}
q^{d(i-1, w)}T_{xws_k\cdots s_i}T_{s_{i-1}} + (1-q)\sum_{i \leq j \leq k, w(j) = j} q^{d(j, w)} T_{xws_k\cdots \hat s_j \cdots s_i}T_{s_{i-1}}.
\end{equation}

Let $j$ be some integer such that $i \leq j \leq k$ and $w(j) = j$, and consider the element $xws_k\cdots \hat s_j \cdots s_i$. This element maps $i$ to $j+1$, and $i-1$ to $xw(i-1)$. But $w(i-1) \leq i-1$, since $w$ is good, so $xw(i-1) \leq i$. However, $i < j+1$. Therefore, $xws_k\cdots \hat s_j \cdots s_is_{i-1}$ is longer than $xws_k\cdots \hat s_j \cdots s_i$, and thus
\[
T_{xws_k\cdots \hat s_j \cdots s_i}T_{s_{i-1}} = T_{xws_k\cdots \hat s_j \cdots s_{i-1}}.
\]

Now, consider $T_{xws_k\cdots s_i}T_{s_{i-1}}$. We see that $xws_k\cdots s_i(i) = xw(k+1) = -1$ and that $xws_k\cdots s_i(i-1) = xw(i-1)$. Since $w$ is good, we have two cases: $w(i-1) < 0$ and $w(i-1) = i-1$.

In the first case, $xw(i-1) < -1 = xw(k+1)$, so $xws_k\cdots s_is_{i-1}$ is longer than $xws_k\cdots s_i$. Thus,
\[
T_{xws_k\cdots s_i}T_{s_{i-1}} = T_{xws_k\cdots s_is_{i-1}}.
\]
In this case, we have overall that
\[
T_xT_wT_{s_k\cdots s_i} T_{s_{i-1}} = q^{d(i-1, w)}T_{xws_k\cdots s_is_{i-1}} + (1-q)\sum_{i \leq j \leq k, w(j) = j} q^{d(j, w)} T_{xws_k\cdots \hat s_j \cdots s_is_{i-1}}.
\]
But $w(i-1) \neq i-1$, so $d(i-1, w) = d(i-2, w)$. Making this substitution, we get Equation \ref{eqn:claim} for parameter $i-1$, completing the inductive step in this case.

In the second case, $xw(i-1) = i > -1$, so $xws_k\cdots s_is_{i-1}$ is shorter than $xws_k\cdots s_i$. This means that
\[
T_{xws_k\cdots s_i}T_{s_{i-1}} = qT_{xws_k\cdots s_is_{i-1}} + (1-q)T_{xws_k\cdots s_i\hat s_{i-1}}.
\]
But $d(i-2, w) = d(i-1, w)+1$ in this case, so using the above equality in Equation \ref{eq:ind_step_partial_product} gives Equation \ref{eqn:claim} for $i-1$, and completing the induction in all cases.

Now, take $i = 1$ in Equation \ref{eqn:claim}, and consider $T_xT_wT_{x^{-1}} = T_xT_wT_{s_k\cdots s_1}T_t$. This is 
\[
q^{d(0, w)}T_{xwx^{-1}t}T_t + (1-q)\sum_{\substack{1 \leq j \leq k \\ w(j) = j}} q^{d(j, w)} T_{xwx_j^{-1}t}T_t.
\]
For $1 \leq j \leq k$ with $w(j) = j$, the element $xwx_j^{-1}t$ maps $1$ to $j+1$, and thus $xwx_i^{-1}$ is longer than $xwx_i^{-1}t$. Therefore, $T_{xwx_j^{-1}t}T_t = T_{xwx_j^{-1}}$. Meanwhile, $xwx^{-1}t$ maps $1$ to $-1$, so $xwx^{-1}$ is shorter than $xwx^{-1}t$. Therefore, $T_{xwx^{-1}t}T_t = pT_{xwx^{-1}} + (1-p)T_{xwx^{-1}t}$. Substituting in these expressions, we find that
\[
T_xT_wT_{x^{-1}} = pq^{d(0, w)}T_{xwx^{-1}} + (1-p)q^{d(0, w)}T_{xwx^{-1}t} + (1-q)\sum_{\substack{1 \leq j \leq n \\ w(j) = j}} q^{d(j, w)}T_{xwx_j^{-1}}.
\]
As $a(w) = d(0, w)$, we find the desired expression.\end{proof}

Note that the $T_{w'}, w' \in B_{k + 1}$, appearing with nonzero coefficient in the sum in Lemma \ref{lem:conj-by-tx} are exactly those in $\Succ(w)$. Further, using Lemma \ref{lem:abc-succ}, we see that Lemma \ref{lem:conj-by-tx} can be rephrased as follows:

\begin{corollary}\label{cor:really-obvious}
We have that 
\[
T_xT_wT_{x^{-1}} = \sum_{w' \in \Succ(w)} p^{\alpha(w')} (1-p)^{\beta(w')} q^{\gamma(w')} (1-q)^{\delta(w')} T_{w'},
\]
where $\alpha(w') = \frac{k+1+a(w')-a(-w')}{2} - \frac{k+a(w)-a(-w)}{2}$, $\beta(w')= a(-w') - a(-w)$, $\gamma(w') = c(w') - c(w)$, and $\delta(w') = \frac{k+1-a(w')-a(-w')}{2} - \frac{k-a(w)-a(-w)}{2}$.
\end{corollary}

We can now prove Theorem \ref{thm:w0k-decomposition}.

\begin{proof}[Proof of Theorem \ref{thm:w0k-decomposition}]
The proof is by induction on $k$, with $k = 1$ as the base case.  The base case is trivial: for $k = 1$ we have $w_{0, k} = t$ and therefore $T_{w_{0, k}}^2 = pT_1 + (1-p)T_t$, matching the desired expression.

Now, assume for the sake of induction that this theorem holds for some fixed $k \geq 1$. Then, note that $w_{0, k+1} = w_{0, k}x^{-1} = xw_{0, k}$, where $x = ts_1\cdots s_k$ is as above, and that $\ell(w_{0, k+1}) =\ell(w_{0, k}) + \ell(x)$.

Then, by the inductive hypothesis,
\[
T_{w_{0, k+1}}^2 = T_xT_{w_{0, k}}^2T_{x^{-1}} = \sum_{w \in G_k} p^{\frac{k+a(w)-a(-w)}{2}}(1-p)^{a(-w)}q^{c(w)}(1-q)^{\frac{k-a(w)-a(-w)}{2}}T_xT_wT_{x^{-1}}.
\]
By Corollary \ref{cor:really-obvious}, we have that for each $w \in G_k$,
\[
T_xT_wT_{x^{-1}} = \sum_{w' \in \Succ(w)} p^{\alpha(w')} (1-p)^{\beta(w')} q^{\gamma(w')} (1-q)^{\delta(w')} T_{w'},
\]
where $\alpha(w') = \frac{k+1+a(w')-a(-w')}{2} - \frac{k+a(w)-a(-w)}{2}$, $\beta(w')= a(-w') - a(-w)$, $\gamma(w') = c(w') - c(w)$, and $\delta(w') = \frac{k+1-a(w')-a(-w')}{2} - \frac{k-a(w)-a(-w)}{2}$.  Together, we then have: 
\[
T_{w_{0, k+1}}^2 =  \sum_{w \in G_k} \sum_{w' \in \Succ(w)} p^{\frac{k+1+a(w')-a(-w')}{2}}(1-p)^{a(-w')}q^{c(w')}(1-q)^{\frac{k+1-a(w')-a(-w')}{2}}T_{w'}.
\]
But by Proposition \ref{prop:succ}, $\coprod_{w \in G_k} \Succ(w) = G_{k+1}$, and the theorem follows.\end{proof}

\subsection{A Recurrence for \texorpdfstring{$\triv(z_{w_{0, k}})$}{triv(z\_w\_0,k)}}
Now, we will use Theorem \ref{thm:w0k-decomposition} to establish a recurrence for the action of $z_{w_{0, k}}$ in the trivial representation $\mathbb{C}_\triv$ of $H_q(S_k)$.

Note that $w_{0, k} = -w_0'$, where $w_0'$ is the longest element of $S_k$. Thus, $w \in B_k$ is in the coset $S_kw_{0, k} = -S_kw_0' = -S_k$ if and only if $-w \in S_k$, i.e. if and only if $w(i) < 0$ for each $1 \leq i \leq k$. For $w \in G_k$, this is equivalent to $a(w) = 0$. Recalling that $T_{w_{0, k}}^2 = \sum_{x \in X_{0, k}} z_xT_x$ for some unique coefficients $z_x \in H_q(S_k)$, where $X_{0, k}$ is the set of distinguished right coset representatives of $S_k$ in $B_k$, we can now see from Theorem \ref{thm:w0k-decomposition} that
\[
z_{w_{0, k}}T_{w_{0, k}} = \sum_{\substack{w \in G_k \\ a(w) = 0}} (p(1-q))^{\frac{k-a(-w)}{2}}(1-p)^{a(-w)}q^{c(w)}T_w.
\]
Reindexing this sum using the equality $\{w \in G_k : a(w) = 0\} = \{-w : w \in S_k, w^2 = 1\}$ and multiplying on the right by $T_{w_{0, k}}^{-1}$, we have:
\[
z_{w_{0, k}} = \sum_{\substack{w \in S_k \\ w^2 = 1}} (p(1-q))^{\frac{k-a(w)}{2}}(1-p)^{a(w)}q^{c(-w)}T_{ww_0'}.
\]
Note that for $w \in S_k$ with $w^2 = 1$, a pair $\{i, j\}$ is tidy in $-w$ if and only if $w(j) < i$ and $w(i) < j$. We will call a pair $\{i , j\}$ that is tidy in $-w$ \emph{neat} in $w$. 

Let $\triv: H_q(S_k) \to \mathbb{C}$ to be the trivial character of $H_q(S_k)$, which satisfies $\triv(T_w) = 1$ for each $w \in S_k$.

\begin{definition}
For $k \geq 1$, let $f_k(p, q)$ be the polynomial
\[
\text{\emph{triv}}(z_{w_{0, k}}) = \sum_{\substack{w \in S_k \\ w^2 = 1}} (p(1-q))^{\frac{k-a(w)}{2}}(1-p)^{a(w)}q^{c(-w)}.
\]
\end{definition}

\begin{lemma}\label{lem:z-recurrence}
We have $f_1(p, q) = 1-p$, $f_2(p, q) = 1 - (1+q)p + p^2$, and 
\[
f_k = p(1-q^{k-1})f_{k-2} + (1-p)f_{k-1}
\] for $k \geq 3$.
\end{lemma}

\begin{proof}
The formulas for $f_1$ and $f_2$ follow immediately from the definition of $f_k$.

Now, fix some $k \geq 3$. Let $w \in S_k$, $w^2 = 1$.  Either $w$ fixes $k$ or not.  In the first case, $w$ restricted to the set $\{1, \ldots, k-1\}$ determines an involution $w' \in S_{k-1}$. Clearly, $a(w) = a(w') + 1$. Consider a pair $\{j, j'\}$ of distinct integers with $1 \leq j, j' \leq k-1$. Since $w'(j) = w(j)$ and $w'(j') = w(j')$, this pair is neat in $w'$ if and only if it is neat in $w$. Pairs of the form $\{j, k\}$ with $1 \leq j \leq k-1$ are never neat in $w$, since $w(k)$ is not less than any such $j$, so $c(-w) = c(-w')$. Therefore, 
\[
\sum_{\substack{w \in S_k \\ w^2 = 1 \\ w(k) = k}} (p(1-q))^{\frac{k-a(w)}{2}}(1-p)^{a(w)}q^{c(-w)} = (1-p)f_{k-1}.
\]

Now, consider the case that $w(k) \neq k$.  Let $w(k) = i$. Then, let $g$ be the restriction of $w$ to $\{1, \ldots, k - 1\} \backslash \{i\}$. The element $g$ is an involution on this set, and we will consider the quantities $a(g)$ and $c(g)$ defined as usual, viewing $g$ as an element of $S_{k - 2}$ via the order preserving bijection between $\{1, \ldots, k - 1\} \backslash \{i\}$ and $\{1, \ldots, k - 2\}$.  The element $w$ fixes neither $i$ nor $k$, and we see that $a(w) = a(g)$. For any pair $\{j, j'\}$ of distinct  integers such that $1 \leq j, j' \leq k-1$ and $j, j' \neq i$, we see similarly that $\{j, j'\}$ is neat in $w$ if and only if it is neat in $g$. Pairs $\{i, j\}$ cannot be neat, as $w(i) = k \geq j$ for any $1 \leq j \leq k$. Finally, the pair $\{j, k\}$, where $1 \leq j < k$, is neat in $w$ if and only if $w(k) = i < j$ and $w(j) < k$.  The first condition is met by all those $j$ strictly between $i$ and $k$, and such $j$ also satisfy the second condition $w(j) < k$ because $w^{-1}(k) = i$.  There are $k - i - 1$ such $j$, so we have:
\[
\sum_{i = 1}^{k-1}\sum_{\substack{w \in S_k \\ w^2 = 1 \\ w(k) = i}} (p(1-q))^{\frac{k-a(w)}{2}}(1-p)^{a(w)}q^{c(-w)} = \sum_{i = 1}^{k-1}p(1-q)q^{k-i-1}f_{k-2}.
\]
Combining the two cases above, we get
\[
f_k = (1-p)f_{k-1} + p(1-q)f_{k-2}\sum_{i = 1}^{k-1} q^{k-i-1} = (1-p)f_{k-1} + p(1-q^{k-1})f_{k-2},
\]
as desired.\end{proof}

\subsection{Separated Sets}

\begin{definition}
Fix a positive integer $k$. Call $S$ a \emph{separated $k$-set} if it is a subset of $\{0, 1, \ldots, k-1\}$ such that for each pair of distinct values $i, j \in S$, we have $1 < |i-j| < k-1$. Let $\Sep_k$ be the set of all separated $k$-sets, and $\Sep_k^+$ be the set of all separated $k$-sets that do not contain $0$.  For any integer $r$ and separated $k$-set $S \in \Sep_k$, let $S + r \in \Sep_k$ be the separated $k$-set given by $$S + r := \{i \in \{0, 1, \ldots, k-1\} \mid i \equiv s+r \mod k \text{ for some } s \in S\}.$$\end{definition}

\begin{lemma}\label{prop:fn-separated}
For each $k \geq 1$, we have
\[
f_k = \sum_{S \in \Sep_k} p^{|S|}(1-p)^{k-2|S|}\prod_{s \in S}(1-q^s).
\]
\end{lemma}

\begin{proof}
Let 
\[
g_k = \sum_{S \in \Sep_k} p^{|S|}(1-p)^{k-2|S|}\prod_{s \in S}(1-q^s)
\]
as above. Note that as $(1-q^0) = 0$, we may replace $\Sep_k$ with $\Sep_k^+$ in the above sum without changing its value.

We will prove the equality of $f_k$ and $g_k$ by showing that they each satisfy the recurrence and initial conditions appearing in Lemma \ref{lem:z-recurrence}.

First, consider $k = 1$. There are two separated $1$-sets, namely the empty set and $\{0\}$. But $(1-q^0) = 0$, so $g_1 = (1-p) = f_1$. Now, consider $k = 2$. There are three separated $2$-sets: the empty set, $\{0\}$, and $\{1\}$. These give us $g_2 = (1-p)^2 + p(1-q) = 1 - (1+q)p + p^2 = f_2$.

Next, fix $k \geq 3$.  Let $S \in \Sep_k^+$. Then, either $S$ contains $k-1$ or it does not. In the first case, $S \backslash \{k-1\}$ is a subset of $\{0, \ldots, k-3\}$. Further since $S$ does not contain $0$, it is not possible for two elements of $S \backslash \{k-1\}$ to differ by $k-3$. Therefore, $S \backslash \{k-1\} \in \Sep_{k - 2}^+$. Additionally, for any $S' \in \Sep_{k - 2}^+$ we have $S' \cup \{k-1\} \in \Sep_k^+$. We therefore have:
\[
\sum_{\substack{S \in \Sep_k^+ \\ k-1 \in S}} p^{|S|}(1-p)^{k-2|S|}\prod_{s \in S}(1-q^s) = p(1-q^{k-1})\sum_{S \in \Sep_{k-2}^+} p^{|S|}(1-p)^{k-2|S|}\prod_{s \in S}(1-q^s),
\]
which is $p(1-q^{k-1})g_{k-2}$.

If $S$ does not contain $k-1$, then $S \in \Sep_{k - 1}^+$, and clearly this determines a bijection between $\Sep_{k - 1}^+$ and the set of $S \in \Sep_k^+$ not containing $k - 1$.  In particular, we have:
\[
\sum_{\substack{S \in \Sep_k^+ \\ k-1 \not\in S}} p^{|S|}(1-p)^{k-2|S|}\prod_{s \in S}(1-q^s) = (1-p)\sum_{S \in \Sep_{k-1}^+} p^{|S|}(1-p)^{k-1-2|S|}\prod_{s \in S}(1-q^s),
\]
which is $(1-p)g_{k-1}$.

Combining the above sums, we see that $g_k = p(1-q^{k-1})g_{k-2} + (1-p)g_{k-1}$, the same recurrence appearing for $f_k$, $k \geq 3$, in Lemma \ref{lem:z-recurrence}.  As $f_1 = g_1$ and $f_2 = g_2$, we have $f_k = g_k$ for all $k \geq 1$, as needed.\end{proof}

\begin{lemma}\label{lem:fk-separated}
If $q$ is a primitive $k^{th}$ root of unity, we have:
\[
f_k = \sum_{S \in \Sep_k} p^{|S|}(1-p)^{k-2|S|}.
\]
\end{lemma}

\begin{proof}
The claim holds for $k = 1$, so assume $k \geq 2$.  The map $S \mapsto S + r$ for any integer $r$ determines a bijection $\Sep_k \rightarrow \Sep_k$. Therefore,
\[
f_k(p, q) = \frac{1}{k}\sum_{r = 0}^{k-1} \sum_{S \in \Sep_k} p^{|S|}(1-p)^{k-2|S|}\prod_{s \in S+r}(1-q^s).
\]
Switching the order of summation, we have
\begin{equation}\label{eq:partial-computation-for-fk-separated-lemma}
f_k(p, q) = \sum_{S \in \Sep_k} \frac{1}{k}p^{|S|}(1-p)^{k-2|S|}\sum_{r = 0}^{k - 1}\prod_{s \in S+r}(1-q^s).
\end{equation}
Expanding the product $\prod_s (1-q^s)$ gives
\[
\prod_{s \in S} (1-q^{s+r}) = \sum_{S' \subset S} (-1)^{|S'|}q^{|S'|r + \left(\sum_{s \in S'} s\right)}.
\]
We may switch the order of summation again to obtain
\[
\sum_{r = 0}^{k-1}\prod_{s \in S}(1-q^{s+r}) = \sum_{S' \subset S} (-1)^{|S'|}q^{\left(\sum_{s \in S'} s\right)}\sum_{r = 0}^{k-1}q^{|S'|r}.
\]
But as $S \in \Sep_k$ and $k \geq 2$, we have $|S'| \leq |S| \leq \lfloor k/2 \rfloor$, and thus, as $q$ is a primitive $k^{th}$ root of unity, $\sum_{r = 0}^{k - 1} q^{|S'|r}$ equals $0$ if $|S'| \neq 0$ and $k$ if $|S'| = 0$.  The claim now follows from Equation \ref{eq:partial-computation-for-fk-separated-lemma}.\end{proof}

We will take the convention for binomial coefficients that $\binom{n}{t} = 0$ whenever $n \geq 0 > t$ or $t > n \geq 0$.

\begin{lemma}\label{lem:separated-count}
There are $\binom{k-i}{i} + \binom{k-i-1}{i-1}$ separated $k$-sets with cardinality $i$.
\end{lemma}

\begin{proof}
Identify subsets of $\{0, \ldots, k - 1\}$ with sequence of 0's and 1's of length $k$; the subset $S \subset \{0, \ldots, k - 1\}$ corresponds to the sequence $(\epsilon_0, \ldots, \epsilon_{k - 1})$ where $\epsilon_i = 1$ if $i \in S$ and $\epsilon_i = 0$ otherwise.  Under this identification, the cardinality of a set equals the number of 1's in the associated sequence, and a subset $S$ is in $\Sep_k$ if and only if it has no adjacent 1's and does not both begin and end with a 1.  Given a 01-sequence of length $k - i$ containing exactly $i$ 1's, the sequence obtained by replacing all 1's with the sequence 01 determines a sequence corresponding to a separated $k$-set of cardinality $i$ not containing 0, and clearly this is a bijection.  Given a 01-sequence of length $k - i - 1$ containing exactly $i - 1$ 1's, the sequence obtained by appending a 1 on the left and replacing all 1's with 01 corresponds to a separated $k$-set of cardinality $i$ containing 0, and clearly this is a bijection as well.  The claim follows.\end{proof}

\begin{lemma}\label{lem:binom}
Let $k \geq i \geq 0$. Then,
\[
\sum_{j = 0}^i (-1)^j\binom{k-2j}{i-j}\binom{k-j}{j} = 1.
\]
\end{lemma}

\begin{proof}  First, note that $\binom{k-2j}{i-j}\binom{k-j}{j}$ equals the trinomial coefficient $$\binom{k-j}{j, i-j, k-i-j} := \frac{(k - j)!}{j!(i - j)!(k - i - j)!},$$ and let $S_{k, i}$ denote the value of the sum on the lefthand side in the lemma statement.

We will prove the lemma by induction on $k$ and $i$ with base cases $k = i$ and $i = 0$. For $i = 0$, we have
\[
S_{k, 0} = \binom{k}{0, 0, k} = 1.
\]
In the case $k = i$, the term $\binom{k-j}{j, i-j, k-i-j}$ is only nonzero when $j = 0$, as otherwise $k-i-j = -j$ is negative. Thus, $S_{k, k}$ collapses to the single term $\binom{k}{0, k, 0} = 1$.

Now, let $k > i > 0$ and assume for that for any $k' \geq i' \geq 0$ such that $k > k'$ or $i > i'$ we have $S_{k', i'} = 1$. By Pascal's identity we have that $S_{k, i}$ equals
\[
\sum_{j = 0}^i (-1)^j\left[\binom{k-j-1}{j-1, i-j, k-i-j} + \binom{k-j-1}{j, i-j-1, k-i-j} + \binom{k-j-1}{j, i-j, k-i-j-1}\right].
\]
Shifting the index of the leftmost sum by $1$, we find that this is equal to 
\[
S_{k, i} = -S_{k-2, i-1} + S_{k-1, i-1} + S_{k-1, i}.
\]
As $k > i > 0$ it follows that $k-2 \geq i-1 \geq 0$, so by the inductive hypothesis
\[
S_{k, i} = -1 + 1 + 1 = 1,
\]
as needed.\end{proof}

We can now prove Theorem \ref{thm:main} in the case $n = 0$:

\begin{proof}[Proof of Theorem \ref{thm:main} for $n = 0$]  Let $k \geq 2$ and let $q$ be a primitive $k^{th}$ root of unity.  To prove the theorem statement for $n = 0$, we need to show that $f_k(p, q) = 1 + (-p)^k$.

Rephrasing Lemma \ref{lem:fk-separated}, we have:
\[
f_k(p, q) = \sum_{i = 0}^{\lfloor k/2 \rfloor} p^i(1-p)^{k-2i}|\{S \in \Sep_k \mid |S| = i\}|.
\]
By Lemma \ref{lem:separated-count}, $|\{S \in \Sep_k \mid |S| = i\}| = \binom{k-i}{i} + \binom{k-i-1}{i-1}$. Substituting this value and expanding $(1-p)^{k-2i}$ gives:
\[
f_k(p, q) = \sum_{i = 0}^{\lfloor k/2 \rfloor} \sum_{j = 0}^{k-2i} (-1)^jp^{i+j} \binom{k-2i}{j}\left(\binom{k-i-1}{i-1} + \binom{k-i}{i}\right).
\]
For any $l$ in the range $0 \leq l \leq \lfloor k/2 \rfloor$, the aggregate coefficient of $p^l$ in the righthand expression above is
\[
\sum_{i = 0}^l (-1)^{l-i} \binom{k-2i}{l-i}\binom{k-i-1}{i-1} + \sum_{i = 0}^l (-1)^{l-i}\binom{k-2i}{l-i}\binom{k-i}{i}.
\]
By Lemma \ref{lem:binom}, we have
\[
\sum_{i = 0}^l (-1)^{l-i}\binom{k-2i}{l-i}\binom{k-i}{i} = (-1)^l.
\]
Also, note that $\binom{k-2i}{l-i}\binom{k-i-1}{i-1} = \binom{(k-2) - 2(i-1)}{(l-1) - (i-1)}\binom{(k-2) - (i-1)}{i-1}$. Thus, if $k-2 \geq l-1 \geq 0$, we have by Lemma \ref{lem:binom} that
\[
\sum_{i = 0}^l (-1)^{l-i} \binom{k-2i}{l-i}\binom{k-i-1}{i-1} = -(-1)^l.
\]
Therefore, for each $l$ such that $1 \leq l \leq \lfloor k/2\rfloor$, the coefficient of $p^l$ in $f_k(p, q)$ is $0$. At $l = 0$, however, 
\[
\sum_{i = 0}^l (-1)^{l-i} \binom{k-2i}{l-i}\binom{k-i-1}{i-1}
\]
is $0$.

Finally, note that since $f_k(p, q) = \sum_{i = 0}^{\lfloor k/2 \rfloor} c_ip^i(1-p)^{k-2i}$ for coefficients $c_i$ independent of $p$, we have that $(-p)^kf_k(p^{-1}, q) = f_k(p, q)$. Thus, for each $l$ with $\lfloor k/2 \rfloor \leq l \leq k-1$, the coefficient of $p^l$ is $0$, and the coefficient of $p^k$ is $(-1)^k$. In particular, $f_k(p, q) = 1 + (-p)^k$, as desired.\end{proof}

\section{General Case}\label{sec:general}
In this section we will prove Theorem \ref{thm:main} by induction on $n$, with the base case $n = 0$ established in the previous section.

Recall the element $w_{n, k} \in X_{n, k} \subset B_{n + k}$. We have $w_{n, k}(i) = i$ for $1 \leq i \leq n$ and $w_{n, k}(n + i) = -(n+k+1-i)$ for $1 \leq i \leq k$.  In particular, by Lemma \ref{lem:b-length}, $\ell(w_{n, k}) = 2nk + \binom{k+1}{2}$.

Let $c$ be the element $s_{n+1}\cdots s_{n+k} \in B_{n+k+1}$. We have that $c(i) = i$ for $1 \leq i \leq n$, $c(n + k + 1) = n + 1$, and $c(n + i) = n + i + 1$ for $1 \leq i \leq k$.  Furthermore, $cw_{n, k}c^{-1} = w_{n+1, k}$ and $\ell(w_{n+1, k}) = 2nk + 2k + \binom{k+1}{2} = \ell(c) + \ell(w_{n, k}) + \ell(c^{-1})$, so $T_{w_{n+1, k}} = T_cT_{w_{n, k}}T_{c^{-1}}$.

\begin{lemma}\label{thm:tcinv-tc}
$T_{c^{-1}}T_c = q^kT_1 + (1-q)\sum_{i = 1}^k q^{i-1}T_{s_{n+k}\cdots s_{n+i}\cdots s_{n+k}}$.
\end{lemma}

\begin{proof}
We will prove this by induction, showing that, for each $0 \leq i \leq k$, we have that 
\begin{equation}\label{eqn:c}
T_{s_{n+i}\cdots s_{n+1}}T_{s_{n+1}\cdots s_{n+i}} = q^iT_1 + (1-q)\sum_{j = 1}^i q^{j-1}T_{s_{n+i}\cdots s_{n+j}\cdots s_{n+i}}
\end{equation}  The case $i = k$ is the claim of the lemma.

The base case $i = 0$, in which case Equation \ref{eqn:c} reads $T_1 = T_1$, is clear. Now, assume that Equation \ref{eqn:c} holds for some fixed $i \geq 0$.  We then have 
\[T_{s_{n+i+1}\cdots s_{n+1}}T_{s_{n+1}\cdots s_{n+i+1}} = q^iT_{s_{n + i+1}}^2 + (1-q)\sum_{j = 1}^i q^{j-1}T_{s_{n + i+1}}T_{s_{n+i}\cdots s_{n+j}\cdots s_{n+i}}T_{s_{n + i+1}}.
\]
By definition, $T_{s_{n + i+1}}^2 = qT_1 + (1-q)T_{s_{n + i+1}}$. Additionally, $s_{n+i+1}\cdots s_{n+j} \cdots s_{n+i+1}$ is reduced. Thus,
\[
T_{s_{n+i+1}\cdots s_{n+1}}T_{s_{n+1}\cdots s_{n+i+1}} = q^{i+1}T_1 + (1-q)q^iT_{s_{n+i+1}} + (1-q)\sum_{j = 1}^i q^{j-1}T_{s_{n+i+1}\cdots s_{n+j}\cdots s_{n+i+1}},
\]
completing the induction.\end{proof}

Now, note that an element $w \in B_{n+k}$ is in the coset $(B_n \times S_k)w_{n, k}$ if and only if $w(n+i) < -n$ for each $1 \leq i \leq k$. This implies that $|w(i)| \leq n$ for each $1 \leq i \leq n$.

\begin{lemma}\label{lem:baby-succ}
For any $w \in (B_n \times S_k)w_{n, k}$, $\ell(cwc^{-1}) = \ell(c) + \ell(w) + \ell(c^{-1})$ and $cwc^{-1} \in (B_{n+1} \times S_k)w_{n+1, k}$. Further, for any $w \in (B_{n+1} \times S_k)w_{n+1, k}$, we have $c^{-1}wx \not\in B_{n+k}$ for any $x < c$ in the Bruhat ordering, and $c^{-1}wc \in (B_n \times S_k)w_{n, k}$ if and only if $cwc^{-1} \in B_{n+k}$.
\end{lemma}

\begin{proof}
Let $w \in (B_n \times S_k)w_{n, k}$. Then, viewing $w$ as an element of $B_{n+k+1}$, $w(n+k+1) = n+k+1$. Thus, $cwc^{-1}(n+1) = n+1$. Additionally, for any $1 \leq i \leq k$, $cwc^{-1}(n+1+i) = cw(n+i) < -(n+1)$. Thus, $cwc^{-1} \in (B_{n+1}\times S_k)w_{n+1, k}$. Also, note that for any $1 \leq i < j \leq n$, $cwc^{-1}(i) = w(i)$ and $cwc^{-1}(j) = w(j)$, so $i$ and $j$ are inverted in $w$ if and only if they are inverted in $cwc^{-1}$. Similarly, for any $1 \leq i < j \leq k$, $cwc^{-1}(n+i+1) = w(n+i)-1$ and $cwc^{-1}(n+j+1) = w(n+j)-1$, so $n+i$ and $n+j$ are inverted in $w$ if and only if $n+i+1$ and $n+j+1$ are inverted in $cwc^{-1}$. In $w$, each pair $(i, j)$ such that $1 \leq i \leq n < j \leq n+k$ is inverted, whereas in $cwc^{-1}$ this is true for each pair $1 \leq i \leq n+1 < j \leq n+k+1$. Finally, we note that the sum of the negated elements of $cwc^{-1}$ is $k$ greater than the sum of negated elements in $w$. It follows that $\ell(cwc^{-1}) = 2k + \ell(w) = \ell(c) + \ell(w) + \ell(c^{-1})$, giving the first statement.

Next, recall that an element $w$ of $B_{n+k+1}$ is an element of $B_{n+k}$ exactly if it fixes $n+k+1$. Let $x$ be some element of $B_{n+k+1}$ such that $x \leq c$ in the Bruhat ordering. We have $c^{-1}wx \in B_{n+k}$ if and only if $w^{-1}c(n+k+1) = x(n+k+1)$.  Note that $w^{-1}c(n+k+1) = w^{-1}(n+1) \leq n+1$. The element $c$ has a unique reduced expression, given by $c = s_{n + 1} \cdots s_{n + k}$, and therefore if $x$ is strictly less than $c$ in the Bruhat ordering then $x$ is given by a proper subexpression of $s_{n + 1} \cdots s_{n + k}$ and therefore satisfies $x(n + k + 1) > n + 1$.  In particular, $x(n+k+1) \geq n+1$, with equality if and only if $x = c$. From these observations, we see that $c^{-1}wx \in B_{n+k}$ if and only if $x = c$ and $w(n+1) = n+1$. But in this case, $c^{-1}wc \in (B_n \times S_k)w_{n, k}$, as desired.\end{proof}

\begin{lemma}\label{lem:double-coset}
Let $w \in B_{n+k}$, let $w' \in B_n \times S_k$, and let $a_x \in H_{p, q}(B_n \times S_k)$, $x \in X_{n, k}$, be the unique coefficients such that
\[
T_wT_{w'} = \sum_{x \in X_{n, k}} a_xT_x.
\]
Then, $a_{w_{n, k}} = 0$ unless $w \in (B_n \times S_k)w_{n, k}$.
\end{lemma}

\begin{proof}
As $(B_n \times S_k)w_{n, k} = (B_n \times S_k)w_{n, k}(B_n \times S_k)$, this follows immediately from the direct sum decomposition
\[
H_{p, q}(B_n \times S_k) = \bigoplus_{x \in (B_n \times S_k)\backslash B_{n + k}/(B_n \times S_k)} \left( \bigoplus_{w \in x} \mathbb{C}T_w\right)
\]
of $H_{p, q}(B_{n + k})$ as a $H_{p, q}(B_n \times S_k)$-bimodule.\end{proof}

Now, note that $w_{n, k}$ and $c$ both fix each integer $i$ such that $1 \leq i \leq n$. Thus, for each $w \in B_n$, $ww_{n, k} = w_{n, k}w$, and $cw = wc$, with $\ell(ww_{n, k}) = \ell(w) + \ell(w_{n, k})$ and $\ell(wc) = \ell(w) + \ell(c)$. Further, note that for each $1 \leq i \leq k-1$, $w_{n, k}s_{n+i} = s_{n+k-i}w_{n, k}$ and $cs_{n+i} = s_{n+i+1}c$. We have that $\ell(w_{n, k}s_{n+i}) = \ell(w_{n, k}) + 1$ and $\ell(cs_{n+i}) = \ell(c) + 1$ for each $1 \leq i \leq k-1$.

Now, we are ready to prove the main theorem.

\begin{proof}[Proof of Theorem \ref{thm:main}]
The proof is by induction on $n$, with the base case $n = 0$ treated in Section \ref{sec:base-case}.

Assume that the theorem holds for some fixed $n \geq 0$ and $k \geq 2$.  We will show that it holds for $n+1$ and $k$ as well. Let $c = s_{n+1}\cdots s_{n+k}$, as above. Then $T_{w_{n+1, k}} = T_cT_{w_{n, k}}T_{c^{-1}}$, so
\[
T_{w_{n+1, k}}^2 = T_cT_{w_{n, k}}T_{c^{-1}}T_{c}T_{w_{n,k}}T_{c^{-1}}.
\]
From Lemma \ref{thm:tcinv-tc}, we have that $T_{c^{-1}}T_c = q^kT_1 + (1 - q)\sum_i q^{i-1}T_{s_{n+k}\cdots s_{n+i}\cdots s_{n+k}}$. Thus,
\[
T_{w_{n+1, k}}^2 = q^kT_cT_{w_{n, k}}^2T_{c^{-1}} + (1-q)\sum_{i = 1}^k q^{i-1}T_cT_{w_{n, k}}T_{s_{n+k}\cdots s_{n+i}\cdots s_{n+k}}T_{w_{n, k}}T_{c^{-1}}.
\]
Let $a_x \in H_{p, q}(B_{n+1} \times S_k)$, for $x \in X_{n+1, k}$, be the unique coefficients such that
\[
T_cT_{w_{n, k}}^2T_{c^{-1}} = \sum_{x \in X_{n+1, k}} a_xT_x,
\]
and let $f_{i, x} \in H_{p, q}(B_{n + 1} \times S_k)$, for $x \in X_{n+1, k}$, be the unique coefficients such that
\[
T_cT_{w_{n, k}}T_{s_{n+k}\cdots s_{n+i}\cdots s_{n+k}}T_{w_{n, k}}T_{c^{-1}} = \sum_{x \in X_{n+1, k}} f_{i, x}T_x.
\]
By these definitions, $z_{w_{n+1, k}} = q^ka_{w_{n+1, k}} + (1-q)\sum_{i=1}^k q^{i-1}f_{i, w_{n+1, k}}$.

Now, consider $T_cT_{w_{n, k}}^2T_{c^{-1}}$. Let $z_x' \in H_{p, q}(B_n \times S_k)$, $x \in X_{n, k}$, be the coefficients such that 
\[
T_{w_{n, k}}^2 = \sum_{x \in X_{n, k}} z_x'T_x.
\]
Then, $T_cT_{w_{n, k}}^2T_{c^{-1}} = \sum_x T_c z_x'T_xT_{c^{-1}}$. Then, note that $c$ commutes with $B_n$ and for any $w \in \langle s_{n+1}, \ldots, s_{n+k-1}\rangle$, $cw = w'c$ for some $w' \in \langle s_{n+2}, \ldots, s_{n+k}\rangle$, where $\ell(cw) = \ell(c) + \ell(w) = \ell(w') + \ell(c)$. Then, for any $w_1 \in B_n$ and $w_2 \in \langle s_{n+1}, \ldots, s_{n+k-1}\rangle$, $1_\triv \otimes T_cT_{w_1}T_{w_2} = 1_\triv \otimes T_{w_2'}T_{w_1}T_c$ for some $w_2' \in \langle s_{n+2}, \ldots, s_{n+k}\rangle$, and hence $1_\triv \otimes T_cT_{w_1}T_{w_2} = 1_\triv \otimes T_{w_1}T_c.$  Therefore, there exist elements $z_x'' \in H_{p, q}(B_{n+1} \times S_k)$ such that $T_cz_x' = z_x''T_c$ and such that $1_\triv \otimes z_x'' = 1_\triv \otimes z_x'$, viewed as elements of $H_{p, q}(B_{n + 1})$. We have
\[
T_cT_{w_{n, k}}^2T_{c^{-1}} = \sum_{x \in X_{n, k}} z_x''T_cT_xT_{c^{-1}},
\]
and by Lemma \ref{lem:baby-succ} we have $a_{w_{n+1, k}} = z_{w_{n+1, k}}''$. Using the inductive hypothesis, $1_\triv \otimes a_{w_{n+1, k}} = 1_\triv \otimes z_{w_{n,k}}'' = 1_\triv \otimes z_{w_{n, k}}' = 1_\triv \otimes (1+(-p)^k)T_1$.

Now, consider the product
\begin{equation}\label{eqn:this-expr}
T_cT_{w_{n, k}}T_{s_{n+k}\cdots s_{n+i}\cdots s_{n+k}}T_{w_{n, k}}T_{c^{-1}}
\end{equation}
for any $1 \leq i < k$.
Note that $s_{n+k}\cdots s_{n+i}\cdots s_{n+k} = s_{n+i}\cdots s_{n+k}\cdots s_{n+i}$ and these expressions are unique, so the product in (\ref{eqn:this-expr}) is also equal to
\[
T_cT_{w_{n, k}}T_{s_{n+i}\cdots s_{n+k}\cdots s_{n+i}}T_{w_{n, k}}T_{c^{-1}}.
\]
Now, note that $w_{n, k}s_{n+i}\cdots s_{n+k-1} = s_{n+k-i}\cdots s_{n+1}w_{n, k}$, and these expressions are reduced, so (\ref{eqn:this-expr}) is furthermore equal to
\[
T_cT_{s_{n+k-i}\cdots s_{n+1}}T_{w_{n, k}}T_{s_{n+k}}T_{w_{n, k}}T_{s_{n+1}\cdots s_{n+k-i}}T_{c^{-1}}.
\]
Also, note that $cs_{n+k-i}\cdots s_{n+1} = s_{n+k-i+1}\cdots s_{n+2}c$. As before, $\ell(cs_{n+k-i}\cdots s_{n+1}) = \ell(c) + k-i$.  Applying this to the above product, we obtain
\[
T_{s_{n+k-i+1}\cdots s_{n+2}}T_cT_{w_{n, k}}T_{s_{n+k}}T_{w_{n, k}}T_{c^{-1}}T_{s_{n+2}\cdots s_{n+k-i+1}}.
\]
But by the definition of $f_{k, x}$, we see that this equals
\[
T_{s_{n+k-i+1}\cdots s_{n+2}}\sum_{x \in X_{n+1, k}} f_{k, x}T_xT_{s_{n+2}\cdots s_{n+k-i+1}}.
\]
Tensoring over $H_q(S_k)$ on the left with $1_\triv$, we obtain
\[
1_\triv \otimes T_cT_{w_{n, k}}T_{s_{n+k}\cdots s_{n+i}\cdots s_{n+k}}T_{w_{n, k}}T_{c^{-1}} = 1_\triv \otimes \sum_{x \in X_{n+1, k}}f_{k, x}T_xT_{s_{n+2}\cdots s_{n+k-i+1}}.
\]
By Lemma \ref{lem:double-coset} we then have
\[
1_\triv \otimes f_{i, w_{n+1, k}}T_{w_{n+1, k}} = 1_\triv \otimes f_{k, w_{n+1, k}}T_{w_{n+1, k}}T_{s_{n+2}\cdots s_{n+k-i+1}}.
\]
But $T_{w_{n+1, k}}T_{s_{n+2}\cdots s_{n+k-i+1}} = T_{s_{n+k}\cdots s_{n+i+1}}T_{w_{n, k}}$, so we also have 
$$
1_\triv \otimes f_{i, w_{n+1, k}}T_{w_{n+1, k}} = 1_\triv \otimes f_{k, w_{n+1, k}}T_{s_{n+k}\cdots s_{n+i + 1}}T_{w_{n+1, k}}.
$$ $$= 1_\triv \otimes T_{s_{n+k}\cdots s_{n+i + 1}}f_{k, w_{n+1, k}}T_{w_{n+1, k}}$$ $$= 1_\triv \otimes f_{k, w_{n+1, k}}T_{w_{n+1, k}}.$$
For the second equality above, note that while $f_{k, w_{n + 1, k}}$ need not necessarily commute with every $T_w$ for $w \in S_k$, we do have $1_\triv \otimes fT_w = 1_\triv \otimes T_wf = 1_\triv \otimes f$ for any $f \in H_{p, q}(B_{n + 1} \times S_k)$ and any $w \in S_k$, and the same holds after right-multiplication by $T_{w_{n + 1, k}}$.  Right-multiplying by $T_{w_{n + 1, k}}^{-1}$, we obtain $$1_\triv \otimes f_{i, w_{n+1, k}} = 1_\triv \otimes f_{k, w_{n+1, k}}.$$

Combining the results for $a_x$ and $f_{i, x}$, we find that
\[
1_\triv \otimes z_{w_{n+1, k}} = q^k(1_\triv \otimes (1 + (-p)^k)T_1) + (1-q)\left(\sum_{i = 1}^k q^{i-1}\right)1_\triv \otimes f_{k, w_{n+1, k}}.
\]
Since $q$ is a primitive $k^{th}$ root of unity, $q^k = 1$ and $\sum_{i = 1}^k q^{i-1} = 0$. Therefore,
\[
1_\triv \otimes z_{w_{n+1, k}} = 1_\triv \otimes (1 + (-p)^k)T_1,
\]
as desired.\end{proof}

\end{document}